\documentclass{proc-l-hijacked}
\usepackage{amssymb, amsmath, amsthm, hyperref,}

\newtheorem{theorem}{Theorem}[section]

\newtheorem{corollary}[theorem]{Corollary}

\theoremstyle{definition}

\newtheorem{example}[theorem]{Example}

\DeclareMathOperator{\rad}{rad}
\DeclareMathOperator{\soc}{soc}
\DeclareMathOperator{\rank}{rank}
\DeclareMathOperator{\spann}{span}
\DeclareMathOperator{\acc}{acc}

\numberwithin{equation}{section}

\begin{document}

	\title[]{Perturbation and Spectral Discontinuity in Banach Algebras}
	\author{R. Brits}
	\address{Department of Mathematics, University of Johannesburg, South Africa}
	\email{rbrits@uj.ac.za}

\begin{abstract}
 We extend an example of B. Aupetit, which illustrates spectral discontinuity for operators on an infinite dimensional separable Hilbert space, to a general spectral discontinuity result in abstract Banach algebras. This can then be used to show that given any Banach algebra, $Y$, one may adjoin to $Y$ a non-commutative inessential ideal, $I$, so that in the resulting algebra, $A$, the following holds: To each $x\in Y$ whose spectrum separates the plane there corresponds a perturbation of $x$, of the form $z=x+a$ where $a\in I$, such that the spectrum function on $A$ is discontinuous at $z$.
\end{abstract}
	\parindent 0mm
	
	\maketitle	

\section{Introduction} Let $A$ be a unital Banach algebra over $\mathbb C$ with unit $\mathbf 1$ and invertible group $A^{-1}$. For $x\in A$ denote the spectrum of $x$ by $\sigma_A(x):=\{\lambda\in\mathbb C: \lambda-x\notin A^{-1}\}$. A frequent source of trouble in spectral theory is the possible discontinuity of the map $x\mapsto\sigma_A(x)$ where $\sigma_A(x)$ lives in the space of compact subsets of $\mathbb C$ (the metric being the  Hausdorff distance). However, since the spectrum function is upper semi-continuous \cite[Theorem 3.4.2]{aup91} it follows, by a result of Kuratowski, that the set of points of $A$ at which $x\mapsto\sigma_A(x)$ is continuous is a dense $G_\delta$ subset of $A$ \cite[Theorem 3.4.3]{aup91}. It is well-known that if $A$ is commutative, then the spectrum function is uniformly continuous on $A$. In fact a characterizing property of commutative Banach algebras (modulo the radical) is uniform continuity of the spectral radius \cite[Corollary 5.2.3]{aup91}, or the spectral diameter \cite[Theorem 2.4]{gra84}, on $A$.  In the general case Newburgh's Theorem \cite[Theorem 3.4.4, Corollary 3.4.5]{aup91} implies that the spectrum function is continuous at all points of $A$ which have a totally disconnected spectrum. Of course, if the spectrum function is continuous at $x\in A$ then so is the spectral radius (denoted $r_\sigma(x)$). The converse of this is not necessarily true (see the paper \cite{apos78} and Example~\ref{E:2} of this paper). Despite the scarcity of  everywhere continuity of the spectrum in non-commutative cases it is, in practice, not so easy to find points in $A$ at which the spectrum function is discontinuous; in particular, early examples illustrating that this may happen are rather technical and seem to have been furnished on an {\it ad hoc} basis (see for example Kakutani's construction first described in \cite{ric60} ). Using subharmonic techniques, Thomas Ransford \cite{ran00} gives a remarkably simple example of a pair of operators, $S$ and $T$, on $l^2$ such that $r_\sigma(T-\lambda S )$ is discontinuous at almost every $\lambda$ in the unit disk. Ransford's example improves on a related but much earlier example of M\"uller \cite{mul76} who uses combinatorial ideas to show that there exist $S$ and $T$ on $l^2$ such that $r_\sigma(T-\lambda S )$ is discontinuous at $\lambda=0$.

M\"uller and Ransford's results are closely related to Kakutani's example which shows that it is possible for a sequence of nilpotent operators to converge to a non-quasinilpotent operator, thus establishing discontinuity of the spectral radius. M\"uller does however show in \cite[p.594]{mul76}, by a modification of his example \cite[p.593]{mul76},  that discontinuity of the spectral radius is also possible in Banach algebras without non-zero quasinilpotents.  In \cite[p.106]{aup77} Aupetit gives an example of spectral discontinuity arising in a completely different manner: There exist two operators, $T$ and $S$, on $l^2$ and a complex sequence $\lambda_n\rightarrow0$  such that for each $n\in\mathbb N$, $\sigma(T+\lambda_nS)$ is the unit circle but $\sigma(T)$ is the closed unit disk. Since the example says nothing about discontinuity of the spectral radius it appears that Kakutani \emph{et.al.}'s results are somewhat stronger than that of Aupetit (Zem\'{a}nek's comments in \cite[p.584]{zem82} are instructive here). What we want to show in this note is that Aupetit's example is at least more general: The key observations in his example are, firstly, that the spectrum of the perturbation $T+S$ gnaws a hole in the spectrum of $T$ and, secondly, that $S$ is a finite rank operator; this is all that is important, any further particular details concerning $S$ and $T$, as well as the underlying space, are immaterial. The existence of the sequence $\lambda_n$ is implicit if the above two observations could be made; Theorem 2.1 gives then a simple and general criterion for spectral discontinuity to occur, and Theorem 2.2 shows that one may easily arrange for the situation in Theorem 2.1.

To simplify, we shall assume throughout that $A$ is semisimple and further also that $\dim(A)>1$ (the latter requirement will really be implicit later on). A two-sided ideal $I$ of a Banach algebra $A$ is said to be inessential if, for each $x\in I$, $\sigma_A(x)$ is either finite or its terms form a sequence converging to zero. For an abstract semisimple Banach algebra $A$ the most familiar instances of inessential ideals in $A$ are the socle, its closure in $A$, and, more generally, the hull-kernel of the socle. The socle, denoted $\soc(A)$, is a two-sided ideal and is, by definition, the smallest left (or right) ideal containing all minimal left(right) ideals. A minimal left(right) ideal can always be written as a principal ideal, $J=Ap$ (respectively $J=pA$), where $p$ is a minimal idempotent (that is, $pAp$  is a division algebra). In the case where $X$ is a Banach space and $A=\mathcal L(X)$, the Banach algebra of the continuous linear operators on $X$, the socle coincides with the ideal of finite rank operators. It is important to mention that there are many examples of Banach algebras, though obviously not $\mathcal L(X)$, which have $\soc(A)=\{0\}$. On the other hand there exist semisimple commutative Banach algebras $A$ for which $\soc(A)\not=\{0\}$; it is not hard to show that this happens precisely when the character space, $\Delta(A)$, contains a singleton set which is open in the weak$^*$ topology on $\Delta(A)$.

In connection with $\soc(A)$, a useful concept is that of rank: For a semisimple Banach algebra $A$ and $a\in A$ we define \begin{equation}\label{R:1}\rank_A(a)=\sup_{x\in
	A}\# \sigma^\prime_A(xa)=\sup_{x\in
	A}\# \sigma^\prime_A(ax)\leq\infty.
\end{equation}
Here $ \sigma^\prime_A(x)=\sigma_A(x)\backslash\{0\}$ and $\#K$ is the number of distinct elements in a set $K$. If the underlying algebra is clear from the context, we shall agree to drop the subscript $A$ in the aforementioned definitions. It can be shown that the set of finite rank elements of $A$ coincides with $\soc(A)$ \cite[Corollary 2.9]{aupmou96} and that the formula \eqref{R:1} agrees with the classical operator rank in the case $A=\mathcal L(X)$ \cite[p. 118]{aupmou96}. Also, the semisimplicity of $A$ guarantees that $\rank(a)=0\Leftrightarrow a=0$. Thus \eqref{R:1} seems to be a very suitable definition of rank in the case of abstract semisimple Banach algebras. For more properties and applications of this spectral rank the Reader can look at \cite{aupmou96, brese98, brlira06}.

Of particular importance to us are the rank one elements of $A$; it can be shown that if $a\not=0$
\begin{equation}\label{R:2}
\rank(a)=1\Leftrightarrow aAa=\mathbb Ca.
\end{equation}
So minimal idempotents are rank one in the sense of \eqref{R:1}. It follows readily from \eqref{R:2} that if $\rank(a)=1$ then there exists a unique, non-zero, bounded linear functional $\tau_a$ on $A$ satisfying $\tau_a(x)a=axa$ for each $x\in A$. For a rank one element $a\in A$ we shall call this functional the \emph{characteristic functional} of $a$. From the assumption $\dim(A)>1$ and the aforementioned functional relationship it easily follows that a rank one element $a$ has
$\sigma(a)=\{\tau_a(\mathbf 1),0\}$. Another useful identity, which is also easy to verify, is the following: Let $x\in A$ be arbitrary and $a\in A$ be a rank one element such that $ax\not=0$. Then $ax$ has rank one and $\tau_{ax}(\mathbf 1)=\tau_a(x)$ (the same statement holds for $xa$). Our first result, Theorem~\ref{T:1}, which is obtained through an application of this identity, improves on a perturbation theorem of Aupetit:

If $H$ is an infinite dimensional separable Hilbert space, then Fong and Sourour \cite{fonsou84} show that any $T\in\mathcal L(H)$ such that $T\notin\{\lambda I+K:\lambda\not=0,\ K\mbox{ compact}\in\mathcal L(H)\}$ is the sum of two quasinilpotent operators. Their result immediately implies that there exist two quasinilpotent operators $Q_1,Q_2\in\mathcal L(H)$ such that $\sigma(Q_1+Q_2)$ is infinite (in fact uncountable). In \cite[pp. 105-106]{aup91} Aupetit addresses this question for a general Banach space $X$. As he remarks, the problem is now much harder since  $X$ may not possess a topological basis. Using a result of Grabiner \cite[Lemma 5.6.9]{aup91}, which guarantees the existence of a non-nilpotent, quasinilpotent, compact $Q\in\mathcal L(X)$ whenever $\dim(X)=\infty$, together with the subharmonic theory of spectral finiteness \cite[V\S6]{aup91}, Aupetit shows that if $\dim(X)=\infty$, then there exist two
quasinilpotent and compact operators $Q_1,Q_2\in\mathcal L(X)$ such that $\#\sigma(Q_1+Q_2)=\infty$ \cite[Theorem 5.6.10]{aup91}. For the proof of the next result, and also in the remainder of this paper, $X^\prime$ will be the continuous dual of a normed space $X$.

\begin{theorem}\label{T:1} Let $X$ be an infinite dimensional Banach space and let $V\in\mathcal L(X)$ be any non-nilpotent quasinilpotent operator. Then, corresponding to $V$, there exists a rank one operator $Q\in\mathcal L(X)$ with $Q^2=0$ such that
	\[\#\sigma(V+\alpha Q)=\infty\] for all non-zero $\alpha\in\mathbb C$ with at most one exception.
\end{theorem}

\begin{proof} For each $k\in\mathbb N$ define $A_k=\{x\in X:V^kx=0\}$. If $X=\cup_{k\geq1}A_k$, then by Baire's Theorem
	there is $n$ such that $A_n$ contains a non-empty open set of $X$. But, since $V$ is not nilpotent, $A_n$ is a proper vector subspace of $X$ and hence has empty interior in $X$. Thus, we can find $x\in X$ such that $V^kx\not=0$ for each $k\geq1$. Let $X^\prime_x=\{\phi\in X^\prime:\phi(x)=0\}$ be the annihilator of $x$ (which is a non-zero Banach subspace of $X^\prime$), and for each $k\in\mathbb N$ define the closed subspace $A_{x,k}=\{\phi\in X^\prime_x:\phi(V^kx)=0\}$ of $X^\prime_x$. If $X^\prime_x=\cup_{k\geq1}A_{x,k}$, then again by Baire's Theorem
	there is $n$ such that $A_{x,n}$ contains a non-empty open set of $X^\prime_x$. However, since $A_{x,n}$ is a vector space, we must have
	$A_{x,n}=X^\prime_x$ implying that $\phi(V^nx)=0$ for each $\phi\in X^\prime$ with $\phi(x)=0$. But now, since $\{x,V^nx\}$ is linearly independent, another contradiction follows from the  separation properties of the Hahn-Banach Theorem. We may hence conclude that there exists $\phi\in X^\prime_x$ such that $\phi(V^kx)\not=0$ for each $k\geq1$. Using $\phi$ and $x$ we construct $Q\in\mathcal L(X)$ with the desired properties: Define $P\in\mathcal L(X)$ by $Pz=\phi(z)x$ which has rank $1$ and satisfies $P^2=0$. For $\lambda\not=0$ the factorization $\lambda-(V+P)=(\lambda-V)[I-(\lambda-V)^{-1}P]$ implies that $\lambda\in\sigma(V+P)\Leftrightarrow1\in\sigma((\lambda-V)^{-1}P)$.
	Now $\sigma((\lambda-V)^{-1}P)=\{\tau_P((\lambda-V)^{-1}), 0\}$ and the function $\tau_P((\lambda-V)^{-1})$ is holomorphic on $\mathbb C-\{0\}$. The corresponding Laurent series, valid for all $\lambda\not=0$, is given by \[\tau_P((\lambda-V)^{-1})=\sum_{j=0}^\infty\frac{\tau_P(V^j)}{\lambda^{j+1}}.\]
	Now choose any fixed $z\in X$ such that $\phi(z)\not=0$ and notice that $(PV^jP)(z)=\phi(z)\phi(V^jx)x\not=0$ for all $j\geq1$. So, since
	$0\not=PV^jP=\tau_P(V^j)P\Rightarrow\tau_P(V^j)\not=0$ for each $j\geq1$, we see that $0$ is an essential singularity of $\tau_P((\lambda-V)^{-1})$.
	By Picard's Theorem \cite[XII Theorem 4.2]{con86} there exists $0\not=\beta\in\mathbb C$ and a sequence $\lambda_n\rightarrow0$ such that
	$\tau_P((\lambda_n-V)^{-1})=\beta$ for each $n$. If we set $Q=\frac{1}{\beta}P$ and notice that $\tau_Q=\tau_{\frac{P}{\beta}}=\frac{1}{\beta}\tau_P$, then it follows that $\sigma(V+Q)=\{\lambda_1,\lambda_2,\dots\}\cup\{0\}$ and the proof is complete.
\end{proof}

Since a rank one quasinilpotent $Q\in\mathcal L(X)$ always takes the form $Qx=(a\otimes \phi)(x)$ where $0\not=a\in X$ and $0\not=\phi\in X^\prime$ satisfies $\phi(a)=0$, it might not be too hard, in concrete cases, to discover a suitable $Q$ satisfying the conclusion of Theorem~\ref{T:1}:\\
\begin{example}
	Let $X=C[0,2\pi]$ be the Banach space of continuous functions on $[0,2\pi]$ and let $V\in \mathcal L(X)$ be the Volterra operator \[(Vf)(t)=\int_0^tf(x)\,dx,\ \ \ t\in[0,2\pi]\] on $X$. So $V$ is quasinilpotent, but not nilpotent. Let $\phi\in X^\prime$ be defined by $\phi(f)=\int_0^{2\pi}f(t)\,dt$. If we take $g\in X$ as $g(t)=\sin t$ and define $Q\in \mathcal L(X)$ by $(Qf)(t)=\phi(f)g(t), \ f\in X$, then $Q$ is rank one.
	For each $n\geq1$ $V^n$ maps $g$ onto a function which takes the form
	$P_n(t)\pm\cos t$ or $P_n(t)\pm\sin t$ where $P_n$ is a polynomial with rational coefficients and $\deg(P_n)=n-1$.
	So, since $\pi$ is transcendental, it follows that $\phi(V^ng)\not=0,\ n\in\mathbb N$.
	But for each $n\in\mathbb N$,  $\phi(V^ng)=\tau_Q(V^n)$ which consequently proves that $0$ is an essential singularity of  $\tau_Q((\lambda-V)^{-1})$. Now, as in the proof of Theorem~\ref{T:1}, Picard's Theorem implies the existence of infinitely many scalars $\alpha$ such that $\#\sigma(V+\alpha Q)=\infty$.
\end{example}

\section{ Perturbation and spectral discontinuity}

If $\rho(x)\subset\mathbb C$ denotes the resolvent set of $x\in A$ then $\rho(x)$ has precisely one unbounded connected component, and at most countably many bounded components in $\mathbb C$. Following Conway \cite[p.206]{con90} we refer to the bounded components of $\rho(x)$, if there are any, as the \emph{holes} of $\sigma(x)$. We denote the polynomially convex hull of $\sigma(x)$ by $\sigma^h(x)$ and the set of accumulation points of $\sigma(x)$ by $\acc\sigma(x)$. From the theory of perturbation by inessential elements (for a very nice abstract account of this look at \cite[V\S7]{aup91}) we have the following: Let $I$ be an inessential ideal of $A$ and let $x\in A$, $a\in I$. Then $\acc\sigma(x)\subseteq\sigma^h(x+a)$ and $\acc\sigma(x+a)\subseteq\sigma^h(x)$. One implication of this is that, if $\sigma(x)$ has a hole, say $H\subset\mathbb C$, and $a\in I$, then $\sigma(x+a)$ may fill up the hole $H$ (i.e. $H\cap\sigma(x)=\emptyset$ but $H\subset\sigma(x+a)$). Obviously perturbation by an inessential element may then cause a hole to appear as well (as Aupetit's example illustrates).
The following simple implication, which already appeared in proof of Theorem~\ref{T:1}, will be used throughout the remainder of this paper: For each $x,y\in A$

\begin{equation}\label{R:3}\lambda\notin\sigma(x)\Rightarrow\bigl[\lambda\in\sigma(x+y)\Leftrightarrow1\in\sigma((\lambda-x)^{-1}y)\bigr].\end{equation}
With the Scarcity Principle \cite[Theorem 3.4.25]{aup91} and repeated use of pigeonhole arguments we can prove:

\begin{theorem}\label{T:2} Let $A$ be a semisimple Banach algebra, $x\in A$ and suppose $H\subset\mathbb C$ is a hole of\, $\sigma(x)$. Then, for every $a\in\soc(A)$ such that $H\subset\sigma(x+a)$, the spectrum function $z\mapsto\sigma(z)$ is discontinuous at $z=x+a$. Hence if $x\in A$ and there exist $H\subset\mathbb C$ and $a\in\soc(A)$ such that $H$ is a hole of $\sigma(x+a)$ but not a hole of $\sigma(x)$, then the spectrum function is discontinuous at $x$.
\end{theorem}

\begin{proof} In the proof $B(\lambda,r)$, $r>0$ is the usual notation for an open disk in the plane. With the hypothesis, define $f(\lambda)=(\lambda-x)^{-1}a$ which is analytic on the domain $H$. Since $a\in\soc(A)$ we have that $\#\sigma(f(\lambda))<\infty$ for each $\lambda\in H$, whence it follows from the Scarcity Principle that there is $n>0$ and a closed and discrete subset $E$ of $H$ such that $\#\sigma(f(\lambda))=n$ for $\lambda\in H\backslash E$, and
	$\#\sigma(f(\lambda))<n$ for $\lambda\in E$. Moreover, the $n$ points of $\sigma(f(\lambda))$ are locally holomorphic functions on $H\backslash E$. Now either $E$ is finite or it is a countable set $\{\alpha_1,\alpha_2,\dots\}$ all of whose converging subsequences attain their limits on the boundary of $H$. We assume $E$ is countable since the proof for the case where $E$ is finite is similar. The union $\cup_{i=1}^\infty\sigma((\alpha_i-x)^{-1}a)$ being at most countable, implies there exists a sequence $\beta_k\subset(0,1)$ such that $\beta_k\rightarrow1$ and $\frac{1}{\beta_k}\notin\cup_{i=1}^\infty\sigma((\alpha_i-x)^{-1}a)$ for each $k$. To show that the spectrum functions is discontinuous at $x+a$ we will show that $\sigma(x+\beta_ka)\nrightarrow\sigma(x+a)$ as $k\rightarrow\infty$. Since $\lambda\notin\sigma(x)$ for $\lambda\in H$, notice that \eqref{R:3} implies $1\in\sigma(f(\lambda))$ for each $\lambda\in H$, and, moreover, that $\frac{1}{\beta}\in\sigma(f(\lambda))\Leftrightarrow\lambda\in\sigma(x+\beta a)$ holds for all $\beta\not=0$. Fix $\lambda_0\in H\backslash E$. Then there is $r^\prime>0$ and $n$ holomorphic functions on $B(\lambda_0,r^\prime)$ say, $\{\gamma_1,\gamma_2,\dots,\gamma_n\}$ such that $\sigma(f(\lambda))=\{\gamma_1(\lambda),\dots,\gamma_n(\lambda)\}$ for each $\lambda\in B(\lambda_0,r^\prime)\subset H\backslash E$. Let $0<r<r^\prime$ and observe that one of the functions $\gamma_j$ assumes the value 1 at infinitely many $\lambda\in\overline{B}(\lambda_0,r)$, and hence must be constantly 1 on $B(\lambda_0,r^\prime)$. So we may assume $\gamma_1(\lambda)=1$ for all $\lambda\in B(\lambda_0,r^\prime)$. Furthermore, since  $\#\sigma(f(\lambda))=n$ for $\lambda\in H\backslash E$,  none of the functions $\gamma_2,\dots,\gamma_n$ assumes the value 1 on $B(\lambda_0,r^\prime)$. If $\beta\not=0,1$ is a complex number and $\sigma(x+\beta a)\cap  \overline{B}(\lambda_0,r)$ has infinitely many members, then there are infinitely many $\lambda$'s in $\overline{B}(\lambda_0,r)$ such that $\frac{1}{\beta}\in\sigma(f(\lambda))$. So this means, using the same argument as for $\gamma_1$, that one of $\gamma_2,\dots,\gamma_n$ is constantly $\frac{1}{\beta}$ on $B(\lambda_0,r^\prime)$. Thus, since the sequence $\beta_k$ is infinite and the functions $\gamma_i$ is a finite set, we can find $M$ sufficiently large such that for each $k>M$ the set $L_k=\sigma(x+\beta_ka)\cap\overline{B}(\lambda_0,r)$ is finite. Suppose $\cup_{k>M}L_k$ is infinite. Then, without loss of generality, we may assume that each $L_k,\ (k>M)$ contains at least one element
	say $\lambda_k$, so that $(\lambda_k)$ forms a sequence of distinct elements. So arguing as above (and passing to a subsequence of $(\lambda_k)$ if necessary) we see that there is some fixed $j\in\{2,\dots,n\}$ such that $\gamma_j(\lambda_k)=\frac{1}{\beta_k}$. But, being bounded, we may assume $(\lambda_k)$ (or otherwise a subsequence thereof) converges. Of course the limit, say $\lambda^\prime$, belongs to $\overline{B}(\lambda_0,r)$. By continuity of $\gamma_j$ on $B(\lambda_0,r^\prime)$ it follows that $\gamma_j(\lambda^\prime)=\lim_{k\rightarrow\infty}\frac{1}{\beta_k}=1$ which contradicts the fact that none of the functions $\gamma_2,\dots,\gamma_n$ assumes the value 1 on $B(\lambda_0,r^\prime)$. Thus there is $N>M$ such that for all $k>N$ we have $L_k$ is constant and finite, which in turn implies the existence of $B(\alpha_0,\epsilon)\subset B(\lambda_0,r)$ such that $\sigma(x+\beta_ka)\cap B(\alpha_0,\epsilon)=\emptyset$ for $k>N$. But $B(\alpha_0,\epsilon)\subset\sigma(x+a)$ and so the spectrum is discontinuous at $x+a$.
\end{proof}

Every Banach algebra $Y$ which contains elements with spectra separating the plane is a source of spectral discontinuities in the following sense: There exists a relatively small superalgebra $A$ of $Y$ such that to each $x\in Y$ whose spectrum separates the plane, there corresponds a rank one element $a\in A$, such that the spectrum function is discontinuous at $x+a$ in the algebra $A$. The idea is to adjoin to $Y$ a non-commutative socle, similar to the way in which one would adjoin an identity element to a non-unital Banach algebra, and then to show that one always reaches the hypothesis of Theorem~\ref{T:2} in the algebra $A$.

\begin{theorem}\label{T:3} Let $Y$ be a semisimple Banach algebra. Then there is a semisimple Banach algebra $A$ with the following properties:
	\begin{itemize}
		\item[(a)]{ $A$ is a norm-preserving and spectrum-preserving extension of $Y$.}
		\item[(b)]{ $A$ contains a closed inessential ideal $I$ such that $A/{I}$ is isometrically isomorphic to $Y$.  }
		\item[(c)]{ For each $a\in Y$ such that $\sigma(a)$ separates the complex plane there is $w\in I$ such that $a
			+w$ is a point of spectral discontinuity in $A$. }
	\end{itemize}
\end{theorem}

\begin{proof} We first prove that if $\sigma_Y(a)$ has a hole $H$ then corresponding to the left multiplication operator $L_a\in\mathcal L(Y)$ there exists a rank one operator $P\in\mathcal L(Y)$ such that $\sigma_{\mathcal L(Y)}(L_a+P)$ fills the hole $H$:
	Without loss of generality we may assume $0\in H$. Notice that $\sigma_Y(a^{-1})$ also has a hole containing $0$ which we denote by
	$H^\prime$.  The first step is to show that there exists $\phi\in Y^\prime$ such that
	$\phi(a^{-1})\not=0$ and $\phi(a^{-k})=0$ for $k\geq2$. This would be possible if we can show that $a^{-1}\notin\overline{\spann}\{a^{-2},a^{-3},\dots\}$. If this is not the case then ${\bf 1}=\lim_nP_n$ where
	$P_n$ is a sequence of polynomials of the form $P_n=\alpha_{1(n)}a^{-1}+\cdots+\alpha_{k(n)}a^{-k}$. But for any $k\geq1$
	\begin{align*}
	\|{\bf 1}-(\alpha_1a^{-1}+\cdots+\alpha_ka^{-k})\|&\geq \rho({\bf 1}-(\alpha_1a^{-1}+\cdots+\alpha_ka^{-k}))
	\\&\geq\max_{\lambda\in \partial H^\prime}|1-(\alpha_1\lambda+\cdots+\alpha_k\lambda^{k})|\\&\geq1
	\end{align*}
	where the final inequality follows from the Maximum Principle applied on the domain $H^\prime$. So we can find $\phi\in Y^\prime$ such that
	$\phi(a^{-1})\not=0$ and $\phi(a^{-k})=0$ for $k\geq2$. Obviously we may assume $\phi(a^{-1})=-1$.  Now let $P\in\mathcal L(Y)$ be defined by
	$Px=\phi(x){\bf 1}$. Then $P$ is rank one and the characteristic functional, $\tau_P$, is given by
	$\tau_P(S)=\phi(S{\bf1}),\ S\in \mathcal L(Y)$. From this one calculates, for $k\in\mathbb N$, that $\tau_P(L_{a^{-k}})=\phi(a^{-k})$
	and from the series expansion of $\tau_P((\lambda-L_a)^{-1})$ in a neighborhood of $0$ we see that $\tau_P((\lambda-L_a)^{-1})=1$  for all $\lambda\in H$. Thus \eqref{R:3} implies $\sigma_{\mathcal L(Y)}(L_a+P)$ fills the hole $H$.
	
	Let $J=\soc(\mathcal L(Y))$, $\bar J$ the closure of $J$ in $\mathcal L(Y)$ and consider the collection of formal sums
	\[A=\{a+S:a\in Y, S\in \bar J\}.\] With addition and scalar multiplication defined in the obvious way and multiplication by
	\[(a+S)(b+W)=ab+L_aW+SL_b+SW\] it follows, from the fact $\bar J$ is a two-sided ideal, that $A$ is a unital algebra. Moreover, $$\|a+S\|=\|a\|+\|S\|$$ defines a complete algebra norm on $A$.  Notice that for each $a\in Y$ and each $S\in\bar J$ we have that $\sigma_{\mathcal L(Y)}(L_a+S)\subseteq\sigma_A(a+S)$.
	Suppose $a+S$ belongs to the radical of $A$. If $b+W\in A$ is arbitrary, then we have
	
	\begin{align*}\|[(a+S)(b+W)]^n\|^{\frac{1}{n}}&=\|(ab)^n+R_n\|^{\frac{1}{n}}=\bigl(\|(ab)^n\|+\|R_n\|\bigr)^{\frac{1}{n}}\geq\|(ab)^n\|^{\frac{1}{n}}
	\end{align*}
	for some sequence $R_n$ in $\bar J$. From the semisimplicity of $Y$ it follows that $a=0$. Thus, a radical element of $A$ has the form $0+S$ where $S\in\bar J$. But if this is the case, then for each $W$ in $\mathcal L(Y)$ we have that $0+WSW\in A$ and that $\sigma_{\mathcal L(Y)}((SW)^2)\subseteq\sigma_A(0+S(WSW))=\{0\}$. Thus $\sigma_{\mathcal L(Y)}(SW)=\{0\}$ which implies $S=0$ since $\mathcal L(Y)$ is semisimple.  So  $A$ is semisimple whenever $Y$ is. It is elementary to prove that $\{0+S:S\in J\}\subseteq\soc(A)$ and hence by \cite[Corollary 5.7.6]{aup91} $I=\{0+S:S\in \bar J\}$ is the required inessential ideal in (b). Note here that above containment may be strict which implies that the homomorphism
	$a+S\mapsto L_a+S$ does not necessarily embed $A$ into $\mathcal L(Y)$.
	To prove (c): If $\sigma_Y(a)=\sigma_{\mathcal L(Y)}(L_a)$ separates the plane, then, by the first part of the proof, we can find $P\in J$ such that $\sigma_{\mathcal L(Y)}(L_a+P)$ fills a hole of $\sigma_{\mathcal L(Y)}(L_a)$. But $\sigma_A(a+0)=\sigma_{\mathcal L(Y)}(L_a)$ and $\sigma_{\mathcal L(Y)}(L_a+P)\subseteq\sigma_A(a+P)$ imply that $\sigma_A(a+P)$ fills a hole of $\sigma_A(a+0)$. The result then follows from Theorem~\ref{T:2} since $0+P\in\soc(A)$ and $a+P=(a+0)+(0+P)$.
\end{proof}

The extension $A$ in Theorem 2.2 seems manageable for two reasons: The adjoined inessential ideal $I$ depends only on the continuous  dual of the algebra $Y$, and, secondly, the perturbation theory of inessential elements is a well understood topic. Because of the Gelfand Transformation, the commutative case is, somewhat ironically, a good starting point for constructing spectral discontinuities via Theorem~\ref{T:3}:

\begin{example}\label{E:2} Let $S\subset\mathbb C$ be the unit circle, and let $Y=C(S)$ with the usual sup norm. Denote by $Y_0$ the subalgebra of $Y$ consisting of  complex functions having a holomorphic extension to neighborhoods of $S$. Let $f\in Y_0$, not a constant function, be holomorphic on a neighborhood $N_f$ containing $S$. Then there are at most finitely many $z\in S$ such that $f^\prime(z)=0$, and corresponding to each $z\in S$ such that $f^\prime(z)\not=0$ there is a neighborhood $U_z\subset N_f$ such that $f$ is injective on $U_z$. This shows that $f(S)$ separates the plane, and hence that $\sigma_Y(f)$ has at least one hole. So if we adjoin to $Y$ a non-commutative socle as in Theorem~\ref{T:3} we get the following result: Corresponding to each non-constant $f\in Y_0$ there exists $a\in A$ such that the spectrum function on $A$ is discontinuous at $f+a$. Moreover,  using \cite[Theorem 5.7.4]{aup91}, it is not hard to see that the spectral radius is continuous on $A$.
\end{example}

\section{ Commuting perturbations}

If $a\in\soc(A)$ commutes with $x$ then $\sigma(x+a)$ cannot fill a hole of $\sigma(x)$; if this was possible, then Theorem 2.1 would predict the existence of spectral discontinuities in some commutative algebra which is absurd. More intuitively this should also follow from \cite[Theorem 5.7.4]{aup91}, together with the containment $\sigma(x+a)\subseteq\sigma(x)+\sigma(a)$ which holds whenever $x$ and $a$ commute. We give a sharp estimate, in terms of rank, for the difference between $\sigma(x+a)$ and $\sigma(x)$ where $a\in\soc(A)$ commutes with $x\in A$. For this we shall need a generalization of Aupetit and Mouton's Diagonalization Theorem \cite[Theorem 2.8]{aupmou96}:

\begin{theorem}\label{T:4}[Generalized Diagonalization Theorem] Let $A$ be a semisimple Banach algebra and
	$0\not=a\in\soc(A)$. Then $a$ is a linear combination of mutually orthogonal minimal idempotents if and only if
	there exists $y\in A$ commuting with $a$
	such that $\rank(a)=\#\sigma^\prime(ya)$
\end{theorem}

\begin{proof} Suppose $\rank(a)=\#\sigma^\prime(ay)=n$
	and that $ya=ay$.  We first show that this hypothesis implies that we can
	actually take $y$ invertible. If $a$ is invertible and $b\in A$ is arbitrary, then
	\[\#\sigma^\prime(b)=\#\sigma^\prime(a(a^{-1}b))\leq\rank(a)=n\] which shows that every element of $A$ has finite spectrum. By the  Hirschfeld-Johnson criterion \cite{hirjo72} $A$ is finite dimensional
	so the Wedderburn-Artin theorem forces $y$ invertible. If $0\in\sigma(a)$ then
	$0\in\sigma(ax)$ for all $x\in A$ because $\sigma(ax)$ is finite. Since the function $\lambda\mapsto a(\lambda-y)$ is analytic from
	$\mathbb C$ into $A$, and $0\in\sigma(a(\lambda-y))$ for all
	$\lambda\in\mathbb C$ the Scarcity Principle says that $\{\lambda\in\mathbb
	C:\#\sigma^\prime(a(\lambda-y))<n\}$ is discrete in $\mathbb C$.
	Hence we can find $\lambda$ in the resolvent set of $y$ such that
	$\#\sigma^\prime(a(\lambda-y))=n=\rank(a)$. So without loss of
	generality we may assume $y\in A^{-1}$. By Aupetit and Mouton's
	Diagonalization Theorem there exist mutually orthogonal minimal idempotents
	$p_1,\dots,p_n$ and distinct non-zero scalars
	$\lambda_1,\dots,\lambda_n$  such that
	$ay=\sum_{j=1}^n\lambda_jp_j.$ Since for each $j$ \[p_j=\frac{1}{2\pi
		i}\int\limits_{\Gamma_j}(\lambda-ay)^{-1}\,d\lambda\] where
	$\Gamma_j$ is a small circle surrounding $\lambda_j$ and
	separating $\lambda_j$ from the remaining spectrum of $ay$, we see
	that $y^{-1}$ commutes with $p_j$.
	From the minimality of $p_j$ we get
	\[a=\sum_{j=1}^n\lambda_jp_jy^{-1}=\sum_{j=1}^n\lambda_jp_jy^{-1}p_j=
	\sum_{j=1}^n\lambda_j\beta_jp_j\] with $0\not=\beta_j\in\mathbb C$. Conversely if $a=\sum_{j=1}^n\lambda_jp_j$ where the $p_j$ are minimal
	mutually orthogonal idempotents, then each $p_j$ commutes with $a$. Also from \cite[Theorem 2.16]{aupmou96} we have that $\rank(a)=n$.
	So the result follows if we take $y=\sum_{j=1}^n\frac{\alpha_j}{\lambda_j}p_j$ where the $\alpha_j$ are distinct non-zero elements of $\mathbb C$.
\end{proof}

\noindent Observe that if $y=\bf1$ then Theorem~\ref{T:4} is precisely the
Diagonalization Theorem.

\begin{corollary} Let $A$ be a semisimple Banach algebra and let
	$a\in\soc(A)$. If $x\in A$ commutes with $a$ then
	$\sigma(a+x)$ and $\sigma(x)$ differ by  at most $2\rank(a)$ complex numbers.
\end{corollary}

\begin{proof} If $\sigma(a)=\{0\}$ then
	$\sigma(a+x)=\sigma(x)$. So we may
	assume $a$ is not quasinilpotent. Let $C_{\{a,x\}}$ be the bicommutant
	of $\{a,x\}$ and form (if necessary)
	$B=C_{\{a,x\}}/ \rad(C_{\{a,x\}})$ so that $B$ is  commutative
	and semisimple. If $z\in C_{\{a,x\}}$ and
	$\tilde z=z+\rad(C_{\{a,x\}})$ is the corresponding coset element in
	$B$ then we have that $\sigma_B(\tilde
	z)=\sigma_{C_{\{a,x\}}}(z)=\sigma_A(z)$.
	Thus we may, without loss of generality, assume that
	$\rad(C_{\{a,x\}})=\{0\}$. Obviously $a\in\soc(B)$ and
	$\rank_B(a)\leq\rank_A(a)$. Suppose that $\rank_B(a)=k<\infty$ and
	that $\sigma_B(a+x)$ contains a set of $k+1$ distinct elements
	$\{\lambda_1,\lambda_2,\cdots,\lambda_{k+1}\}$ belonging to $\rho(x)$. Applying \eqref{R:3} we have, for each
	$i\in\{1,\dots,k+1\}$, that
	$1\in\sigma_B((\lambda_i-x)^{-1}a)$. Since $B$ is commutative
	we can write, using Theorem~\ref{T:4},
	$a=\sum_{j=1}^k\alpha_jp_j$ where $\alpha_j$ are non-zero scalars
	and $p_j$ are mutually orthogonal minimal idempotents belonging to
	$B$. It follows that for $i\in\{1,\dots,k+1\}$
	\[\sigma^\prime_B((\lambda_i-x)^{-1}a)
	=\sigma^\prime_B(\sum_{j=1}^k\alpha_jp_j(\lambda_i-x)^{-1}p_j)
	=\bigcup_{j=1}^k\{\alpha_j\tau_{p_j}((\lambda_i-x)^{-1})\},\]
	and hence that there exist
	$k_0\in\{1,2,\dots,k\}$ and $i_0,i_1\in\{1,\dots,k+1\}$  such that
	
	\begin{equation}\label{R:4}\alpha_{k_0}\tau_{p_{k_0}}((\lambda_{i_0}-x)^{-1})=
	\alpha_{k_0}\tau_{p_{k_0}}((\lambda_{i_1}-x)^{-1})=1.
	\end{equation}
	
	The point now is that, since $B$ is commutative, the characteristic linear functional $\tau_p$
	corresponding to a rank one idempotent $p$ is in fact a character of $B$. So \eqref{R:4}
	implies that $\lambda_{i_0}=\lambda_{i_1}$ which contradicts the
	assumption that the $\lambda_j$ are distinct. So in conclusion
	\[\#[\sigma(a+x)\cap \rho(x)]\leq\rank_B(a)\leq\rank_A(a).\]
	But, by the above arguments, one also has
	\[ax=xa\Rightarrow\#[\sigma((x+a)-a)\cap \rho(x+a)]\leq\rank_A(-a)=\rank_A(a).\]
\end{proof}

\bibliographystyle{amsplain}

\begin{thebibliography}{99}
	

	
	\bibitem{apos78} C. Apostol, \emph{The spectrum and spectral radius as functions in Banach algebras} Bull. Pol. Acad. Sci. Math. 26 (1978), 975--978.
	
	
	
	\bibitem{aup91} B. Aupetit,   \emph{A primer on spectral Theory}, Springer-Verlag, New York,
	1991.
	
	
	\bibitem{aup77} B. Aupetit, \emph{Continuit\'{e} et uniforme continuit\'{e} du spectre dans les alg\`{e}bres
		de Banach}, Studia Math. 61 (1977), 99--114.
	
	\bibitem{aupmou96} B. Aupetit and H. du T. Mouton,  \emph{Trace and
		determinant in Banach algebras},
	Studia Math. 121(2) (1996), 115--136.
	
	\bibitem{brese98} M. Bre\v sar and P. \v Semrl, \emph{Finite rank elements in semisimple Banach algebras}, Studia Math. 128(3) (1998), 287--298.
	
	\bibitem{brlira06} R. Brits, L. Lindeboom and H. Raubenheimer, \emph{Rank
		and the Drazin inverse in Banach algebras},
	Studia Math. 177(3) (2006), 211--223.
	
	\bibitem{con90} J.B. Conway, \emph{A course in functional analysis}, 2nd ed., Springer-Verlag, New York, 1990.
	
	\bibitem{con86} J.B. Conway, \emph{Functions of one complex variable}, Springer-Verlag, New York, 1986.
	
	\bibitem{fonsou84} C.K. Fong and A.R. Sourour, \emph{Sums and products of quasi-nilpotent operators}, Proc. Roy. Soc.
	Edinburgh Sect. A (1984), 193--200.
	
	\bibitem{gra84} S. Grabiner, \emph{The spectral diameter in Banach algebras}, Proc. Amer. Math. Soc. 91(1) (1984), 59--63.
	
	\bibitem{hirjo72} R.A. Hirschfeld and B.E. Johnson,  \emph{Spectral Characterizations of
		Finite-Dimensional Algebras},
	Indag. Math. 34 (1972), 19--23.
	
	\bibitem{mul76} V. M\"uller, \emph{On discontinuity of the spectral radius in Banach algebras}, Comment. Math. Univ. Carolin. 18 (1977), 591--598.
	
	\bibitem{ran00} T.J. Ransford, \emph{Almost-everywhere discontinuity of the spectral radius},
	Proc. Amer. Math. Soc. 129(3) (2001), 749--751.
	
	\bibitem{ric60} C.E. Rickart, \emph{General theory of Banach algebras}, Van Nostrand, 1960.
	
	\bibitem{zem82} J. Zem\'{a}nek, \emph{Properties of the spectral radius in Banach algebras}, Banach Center Publ. Vol. 8 (1982), 579--595.
	
\end{thebibliography}

\end{document}